\documentclass[11pt]{article} 
\usepackage[latin1]{inputenc}   
\usepackage{amsmath,amsfonts,amssymb,amscd,amsthm,amsbsy}

\usepackage{amsfonts,amsmath,layout}
\usepackage{color}

\newcommand{\Z}{\mathbb Z}

\newcommand{\R}{\mathbb R}

\def\R{\mathbb R}

\def\Z{\mathbb Z}

\def\ep{\epsilon}
\def\rg{\rangle} 
\def\lg{\langle}

\newcommand{\be}{\begin{equation}}
\newcommand{\ee}{\end{equation}}
\def\1{{\bf 1}}

\def\m{\noalign{\medskip}}

\def\ds{\displaystyle}
\newcommand{\T}{\mathbb{T}}

\newtheorem{Theorem}{Theorem}[section]
\newtheorem{Definition}[Theorem]{Definition}

\newtheorem{Lemma}[Theorem]{Lemma}

\newtheorem{Remark}[Theorem]{Remark}
\newtheorem{Example}[Theorem]{Example}
\begin{document}

\title{Long time average of  first order mean field games and weak KAM theory}
\author{P. Cardaliaguet\thanks{Ceremade, Universit\'e Paris-Dauphine,
Place du Maréchal de Lattre de Tassigny, 75775 Paris cedex 16 (France)} }

\maketitle

\begin{abstract} We show that the long time average of solutions of first order mean field game systems in finite horizon is governed by an ergodic system of mean field game type. The well-posedness of this later system and the uniqueness of the ergodic constant rely on weak KAM theory. 
\end{abstract}

\section*{Introduction}

The aim of this paper is to study the link between the finite horizon first order mean field game system
\be\label{MFG}
\left\{\begin{array}{cl}
(i)& -\partial_t u^T  +H(x,Du^T) =F(x,m^T(t))\; {\rm in }\; (0,T)\times \R^d\\
\m
(ii) & \partial_t m^T  -{\rm div} (m^TD_pH(x,Du^T))=0\; {\rm in }\; (0,T)\times \R^d\\
\m
(iii)& m^T(0)=m_0, \; u^T(x,T)=u^f(x)\; {\rm in}\; \R^d
\end{array}\right.
\ee
and the  ergodic first order mean field game system
\be\label{eq:ergpb}
\left\{\begin{array}{rl}
(i)& \bar \lambda +H(x, D\bar u) =F(x,\bar m)\; {\rm in}\; \R^d\\
\m
(ii) & -{\rm div} (\bar mD_pH(x,D\bar u))=0\; {\rm in}\; \R^d\\
\m
(ii) & \!\!\!  \ds \int_{Q} \bar u\,dx=0\,,\quad \int_{Q} \bar m\, dx=1
\end{array}\right.
\ee
Let us recall that mean field games have been introduced simultaneously by Lasry and Lions  \cite{LL06cr1, LL06cr2, LL07mf} and by Huang, Caines and Malham\'e \cite{HCMieeeAC06} to analyze large population stochastic differential games. In \eqref{MFG}, the scalar unknowns  $u^T,m^T$ are defined on $[0,T]\times \R^d$ and $F$ is a coupling between the two equations. The function $u^T$ can be understood as the value function---for a typical and small player---of  a finite horizon optimal control problem  in which the density $m^T$ of the other players enters as a data. The optimal feedback of this small player is then given by $-D_pH(x,Du(x))$. When all players play according to this rule, their density $m^T=m^T(t,x)$ evolves in time by equation \eqref{MFG}-(ii).  The ergodic problem \eqref{eq:ergpb} has similar interpretation, expect that now the optimal control problem is of ergodic type: the unknowns are the ergodic constant  $\bar \lambda$, the value function $\bar u$ of the ergodic problem and $\bar m$ which is an associate invariant measure. 

In analogy with the case of optimal control problems, it is expected that, as $T\to+\infty$, the solution of the finite horizon system \eqref{MFG} somehow converges to the solution of the ergodic system \eqref{eq:ergpb}. For second order mean field game systems (i.e., systems corresponding to stochastic control problems with a nondegenerate diffusion) this kind of behavior has been first discussed in \cite{LLperso} and then developed and sharpened in \cite{CLLP,CLLP2}. Results in the discrete setting were also obtained in \cite{GMS}. Typically it is proved in the above mentioned papers that $u^T(0,\cdot)/T$ converges to $\bar \lambda$ while  (a rescaled version of) $m^T$ converges to the invariant measure $\bar m$. The precise meaning of the convergence depends on the coupling $F$, which can be of local or nonlocal nature: when $F$ is local (i.e., $F(x,m(t))=\bar F(x,m(t,x))$ depends on the value $m(t,x)$ of the density of $m(t)$), the convergence holds in Lebesgue spaces. When the coupling $F$ is of nonlocal nature and smoothing, the convergence is uniform. The main result of \cite{CLLP,CLLP2} is an exponential convergence rate when the coupling---and the diffusion---are ``strong enough". 

Here we consider the same issue for first order mean field games in which the coupling $F$ is nonlocal and smoothing. We also assume that all functions are periodic in space and set $\T^d=\R^d/\Z^d$. In the uncoupled case $F=0$, this problem has been the object of a lot of attention in the recent years under the name of weak-KAM theory (see in particular the pioneering works \cite{Fathi1, Fathi2, NR99, Ro98} and the monograph \cite{Fathi3}). Compared to the second order setting, several interesting issues arise: even for $F=0$, one cannot expect the ergodic system to have a {\it unique} solution. As a consequence, in the presence of the coupling, {\it it is not even clear that the ergodic constant $\bar \lambda$ is unique}. There is also a strong difficulty to give a meaning to \eqref{eq:ergpb}-(ii): indeed, as a solution of a classical Hamilton-Jacobi equation, the map $\bar u$ is at most Lipschitz continuous (actually semiconcave); on another hand, the ergodic measure cannot be expected to have a density (again, this is what happens in general for $F=0$): as a consequence {\it the term $\bar mD_pH(x,D\bar u)$ is a priori ill-defined}. To overcome these difficulties, we use two tools: the first one is a typical regularity property arising in weak-KAM theory \cite{Fathi1, Fathi2, Fathi3}: the measure $\bar m$ happens to concentrate on the so-called  projected Mather-set, in which the derivative of $\bar u$ exists and is uniquely defined (independently of the solution $\bar u$). This allows to give a meaning to the term $\bar mD_pH(x,D\bar u)$ and provides a key tool for showing the existence of solutions to \eqref{eq:ergpb}. As for the uniqueness issue, we introduce a weak coercivity condition on the coupling: namely we assume that there is a constant $\bar c>0$ such that, for any pair of measures $m_1, m_2$, 
$$
\int_{\T^d} (F(x,m_1)-F(x,m_2))d(m_1-m_2) \geq \overline c \int_{\T^d} (F(x,m_1)-F(x,m_2))^2 dx
$$
This condition---which is quite natural in the context of mean field game theory (see Example \ref{Ex:1} below)---entails the uniqueness of the ergodic constant. 
Then, using energy estimates, we prove our main result concerning the convergence of  $u^T(0,\cdot)/T$ to $\bar \lambda$: there is a constant $C$ such that  
$$
\sup_{t\in [0,T]} \left\| \frac{u^T(t,\cdot)}{T}-\bar \lambda \left(1-\frac{t}{T}\right) \right\|_\infty \leq  \frac{C}{T^{\frac12}}\;.
$$
As for the convergence of $m^T$, we have little information due to the lack of uniqueness of the ergodic measure $\bar m$. However,  the coupling $F(\cdot,\bar m)$  turns out  to be unique and therefore it seems the good quantity to look at: indeed we have the following estimate: 
$$
\int_0^T \|F(\cdot,m^T(t))-F(\cdot,\bar m)\|_\infty\ dt \leq CT^{\frac12}\;.
$$
Note that this inequality means that $F(\cdot,m^T(t))$ must be close to $F(\cdot,\bar m)$ on a large amount of time. \\

The paper is organized as follows: in a first part we introduce the notations and state the assumptions used throughout the paper. Then we study the ergodic mean field game system. In the last section we prove our convergence result. In appendix we recall the main steps of the proof for the well-posedness of \eqref{MFG}. \\

{\bf Acknowledgement: } We wish to thank Yves Achdou for fruitful discussions. 

This work has been partially supported by the Commission of the
European Communities under the 7-th Framework Programme Marie
Curie Initial Training Networks   Project SADCO,
FP7-PEOPLE-2010-ITN, No 264735, and by the French National Research Agency
 ANR-10-BLAN 0112 and ANR-12-BS01-0008-01.

\section{Preliminaries}

Throughout this note, we work in the periodic setting: we denote by $\T^d$ the torus $\R^d/\Z^d$, by $P(\T^d)$ the set of Borel probability measures on $\T^d$,  and by $P(\T^d\times\R^d)$ the set of Borel probability measures on $\T^d\times \R^d$. Both sets $\T^d$ and by $P(\T^d\times \R^d)$ are endowed with the weak-* convergence. Let us recall that $P(\T^d)$ is compact for this topology.  It will be convenient to put a metric on $P(\T^d)$ (which metricizes the weak-* topology): we shall work with the Monge-Wasserstein distance defined, for any $\mu,\nu \in P(\T^d)$, by
\be\label{MWdistance}
{\bf d}_1(\mu,\nu)= \sup_h \left\{ \int_{\T^d} h d(\mu-\nu) \right\}
\ee
where the supremum is taken over all the maps $h:\T^d\to \R$ which are 1-Lipschitz continuous. 

The maps $H$ and $F$ are periodic in the space arguments: $H:\T^d\times \R^d\to \R$ while $F:\T^d\times P(\T^d)\to\R$. In the same way, the initial and terminal conditions $m_0$ and $u^f$---which are fixed throughout the paper, are periodic in space: $u^f:\T^d\to \R$ is supposed to be of class ${\mathcal C}^2$ while $m^f$ belongs to $P(\T^d)$ is assumed to have a bounded density ($m^f\in L^\infty$). 
 
We now state our key assumptions on the data: these conditions are valid throughout the paper. The map $F$ is supposed to be regularizing:
\be\label{regucondF}
\mbox{\rm The map $m\to F(\cdot,m)$ is Lipschitz continuous from $P(\T^d)$ to ${\mathcal C}^2(\T^d)$.}
\ee
In particular, as $P(\T^d)$ is compact, there is $\bar C>0$ such that
\be\label{boundC2}
\sup_{m\in P(\T^d)} \left\| F(\cdot,m)\right\|_{{\mathcal C}^2} \leq \bar C
\ee
As explained in  the introduction, our key assumption is the following weak coercivity for the coupling:   there is a constant $\overline c>0$ such that, for any $m_1,m_2\in P(\T^d)$,  
\be\label{hyp:Fstrict}
\int_{\T^d} (F(x,m_1)-F(x,m_2))d(m_1-m_2) \geq \overline c \int_{\T^d} (F(x,m_1)-F(x,m_2))^2 dx.
\ee
An example of map $F$ satisfying \eqref{regucondF} and \eqref{hyp:Fstrict} is given below. Concerning the Hamiltonian, 
we suppose that $H$  is of class ${\mathcal C}^2$ on $\T^d\times \R^d$ and quadratic-like in the second variable: 
\be\label{hyp:unifCv}
H\in {\mathcal C}^2(\T^d\times \R^d)\; {\rm and }\; \frac{1}{\bar C} I_d\leq D^2_{pp} H(x,p) \leq \bar C I_d \qquad \forall (x,p)\in \T^d\times \R^d\;.
\ee
Let us recall that, under assumptions \eqref{regucondF} on $F$ and \eqref{hyp:unifCv} on $H$, for any time horizon $T$, there  is a unique solution $(u^T,m^T)$ to the mean field game system \eqref{MFG}: by a solution, we mean that $u^T$ is a Lipschitz continuous  viscosity solution to \eqref{MFG}-(i) while $m^T\in L^\infty((0,T)\times \T^d)$ is a solution of \eqref{MFG}-(ii) in the sense of distribution (see \cite{LL07mf} and Theorem \ref{thm:MFG} in appendix). Recall also that the map $t\to m^T(t)$ is weakly continuous as a measure on $\T^d$. In particular, for any continuous map $\phi:\T^d\to\R$, the integral $\ds \int_{\T^d}\phi(x)m^T(t,x)$ is continuous in $t\in [0,T]$. 

\begin{Example}\label{Ex:1} Assume that $F:\T^d\times P(\T^d)\to \R$ is of the form
$$
F(x,m)= \left(\bar F(\cdot,  m \star\xi (\cdot)) \star \xi\right)(x)
$$
where $\xi:\R^d \to \R$ is a smooth, even kernel with compact support and where $\bar F: \T^d\times \R\to \R$ is a smooth map for which there is a constant $c>0$ with
\be\label{hypex}
c \leq \frac{\partial \bar F(x,z)}{\partial z} \leq \frac{1}{c} \qquad \forall (x,z)\in \T^d\times \R\;.
\ee
Then $F$ satisfies conditions  \eqref{regucondF} and \eqref{hyp:Fstrict}. 
\end{Example}

\begin{proof} The Lipschitz continuity assumption~\eqref{regucondF} is straightforward from the smoothness assumption on $\bar F$. 
Let us now check~\eqref{hyp:Fstrict}. We have, on the one hand,
$$
\begin{array}{l}
\ds \int_{\T^d} (F(x,m_1)-F(x,m_2))d(m_1-m_2)  \\ 
\qquad \ds =\;  \int_{\T^d} [\bar F(\cdot,m_1\star \xi)-\bar F(\cdot,m_2\star \xi )]\star\xi\  d(m_1-m_2)\\
\m 
\qquad = \; \ds \int_{\T^d} [\bar F(\cdot,m_1\star \xi)-\bar F(\cdot,m_2\star \xi )](m_1\star\xi-m_2\star\xi)\\ \m
\end{array}
$$
which, in view of our growth condition on $\bar F$, implies that
\be\label{et1Exemple}
\begin{array}{rl}
\ds \int_{\T^d} (F(x,m_1)-F(x,m_2))d(m_1-m_2) \geq   
& \ds  c \int_{\T^d} (m_1\star \xi-m_2\star \xi )^2
\end{array}
\ee
On another hand, 
$$
\begin{array}{l}
\ds \int_{\T^d} (F(x,m_1)-F(x,m_2))^2 dx\\
\qquad = \; \ds \int_{\T^d} \left[(\bar F(\cdot, m_1\star \xi)-\bar F(\cdot,m_2\star \xi))\star \xi\right]^2 dx \\ \m
\qquad  \leq \; \ds \ds \left\|(\bar F(\cdot, m_1\star \xi)-\bar F(\cdot,m_2\star \xi))\star \xi\right\|_\infty^2\\ \m
\qquad \leq \; \ds \|\xi\|^2_{L^2(\R^d)}\left\|(\bar F(\cdot, m_1\star \xi)-\bar F(\cdot,m_2\star \xi))\right\|^2_{L^2(\T^d)}\\ \m
\qquad \leq \; \ds \|\xi\|^2_{L^2(\R^d)} c^{-2} \int_{\T^d} (m_1\star \xi-m_2\star \xi)^2 dx
\end{array}
$$
where we have used the second inequality in the right-hand side of assumption \eqref{hypex} in the last line. Using \eqref{et1Exemple} we deduce that  $F$ satisfies \eqref{hyp:Fstrict} with $\bar c= c^3 \|\xi\|^{-2}_{L^2(\R^d)}$. 
\end{proof}

\section{The ergodic problem}

In this section, we show that the ergodic system \eqref{eq:ergpb} is well-posed. 

\begin{Definition} We say that the triple $(\bar \lambda, \bar u, \bar m)$ is a solution of \eqref{eq:ergpb} if $\bar u$ is Lipschitz continuous viscosity solution of \eqref{eq:ergpb}-(i), if $D\bar u(x)$ exists for $\bar m$-a.e. $x\in \T^d$ and if \eqref{eq:ergpb}-(ii) is satisfied in the sense of distribution. 
\end{Definition}

\begin{Remark}\label{rem:1} {\rm
Because of the regularity of $H$ and $F$, the map $\bar u$ is semiconcave (see, for instance, Theorem 3.3 of \cite{Li}). In particular, $D\bar u$ is continuous at  $\bar m$-a.e. $x\in \T^d$.  Note also that $\bar mD\bar u$ is a vector measure. 
}\end{Remark}

Here is our main result concerning system  \eqref{eq:ergpb}.

\begin{Theorem}\label{th:ergopb} Under assumption  \eqref{regucondF} and \eqref{hyp:unifCv}, there is at least one solution of the ergodic problem \eqref{eq:ergpb}. If, moreover, assumption \eqref{hyp:Fstrict} holds, the ergodic constant is unique: more precisely, if $(\bar \lambda_1,\bar u_1,\bar m_1)$ and $(\bar \lambda_2,\bar u_2,\bar m_2)$ are two solutions of \eqref{eq:ergpb}, then $\bar \lambda_1=\bar \lambda_2$ and $F(\cdot,\bar m_1)=F(\cdot,\bar m_2)$. 
\end{Theorem}

%
%
%
%
%
%

\begin{proof} The existence of the solution relies on several aspects of weak-KAM theory, as developed by Fathi in \cite{Fathi1, Fathi2, Fathi3}. Let $L$ be the Fenchel conjugate of $H$ with respect to the last variable: $$\ds L(x,v)=\sup_{p\in \R^d} \lg p,v\rg -H(x,p)\qquad {\rm for }\; (x,v)\in \T^d\times \R^d\;.
$$
 In view of our assumptions, $L$ is of class ${\mathcal C}^2$ and uniformly convex.  Given $m\in P(\T^d)$, we consider the set $E^m$ of measures $\eta$ on $\T^d\times \R^d$ which are invariant under the Lagrangian flow $\phi^m_t=(\gamma(t),\dot \gamma(t))$ defined by 
$$
\left\{\begin{array}{l}
\ds-\frac{d}{dt} D_v \left[ L(x,\dot x)+F(x,m)\right] + D_x \left[ L(x,\dot x)+F(x,m)\right] =0\\
\ds x(0)=x, \; \dot x(0)=v
\end{array}\right.
$$
The set $M(m)$ of the minimizers of the map $$\ds \eta\to   \int_{\T^d\times \R^d}  L(x,v)+F(x,m)\ d\eta(x,v)$$ over $E^m$ 
 is nonempty, compact and convex subset of $P(\T^d\times \R^d)$. If $\pi:\T^d\times \R^d\to \T^d$ denotes the canonical projection, we finally set $C(m)=\{\pi\sharp \eta, \; \eta\in M(m)\}$. Then $C(m)$ is a convex, compact nonempty subset of $P(\T^d)$. From the continuity assumption  \eqref{regucondF} on $F$ and the coercivity of $L$,  the set-valued map $m\to C(m)$ has a compact graph. Using Kakutani fixed point theorem, one deduce then that
$C$ has a fixed point: there is $\bar m\in P(\T^d)$ such that $\bar m\in C(\bar m)$. Let $\bar \eta\in M(\bar m)$ be such that 
$\pi\sharp \bar \eta=\bar m$. 
 
 Let $(\bar \lambda, \bar u)$ be such that $\bar u$ is a continuous, periodic viscosity solution of 
 $$
 \bar \lambda +H(x, D\bar u) =F(x,\bar m)\qquad {\rm in }\; \T^d\;.
 $$
Since the Hamiltonian is coercive and satisfies our smoothness conditions, $\bar u$ is Lipschitz continuous and semiconcave.  Following Fathi (\cite{Fathi1}, Corollary 2), 
 $$
 \begin{array}{rl}
\ds \bar \lambda \; = & \ds \min_{\eta\in E^{\bar m}}   \int_{\T^d\times \R^d}  L(x,v)+F(x,\bar m)\ d \eta(x,v) \\
=& \ds  \int_{\T^d\times \R^d}  L(x,v)+F(x,\bar m)\ d\bar \eta(x,v)
\end{array}
$$
Moreover $D\bar u$ exists everywhere on  the support of $\bar m$ and is Lipschitz continuous on this support (\cite{Fathi1}, Proposition 3). It is also known that the  canonical map $\pi:{\rm Spt}(\bar \eta)\to {\rm Spt}(\bar m)$ is one-to-one and its inverse is given by $x\to (x,D_pH(x,D\bar u(x)))$ on $ {\rm Spt}(\bar m)$. In particular, the first component $\gamma^x_t$ of the flow $\phi_t$ satisfies
$$
\frac{d}{dt} \gamma_t^x= D_pH( \gamma_t^x,D\bar u( \gamma_t^x))\; {\rm for }\; x\in {\rm Spt}(\bar m)\;.
$$
This implies that equality $-{\rm div}(\bar m D\bar u)=0$ holds in $\T^d$: indeed, as $\bar \eta$ is invariant under the flow $\phi^{\bar m}_t$, $\bar m$ is invariant under the flow $\gamma^x_t$ and we have, for any test function $f\in {\mathcal C}^\infty(\T^d)$,  
$$
\begin{array}{rl}
\ds 0\; =& \ds \frac{d}{dt} \int_{\T^d} f(\gamma^x_t) d\bar m(x)=  \int_{\T^d} \lg Df(\gamma^x_t), D_pH( \gamma_t^x,D\bar u( \gamma_t^x))\rg d\bar m(x)\\
= & \ds  \int_{\T^d} \lg Df(y), D_pH(y,Du(y))\rg d\bar m(y)\;.
\end{array}
$$

 \bigskip

We now show that uniqueness holds. Let $(\bar \lambda_1,\bar u_1,\bar m_1)$ and $(\bar \lambda_2,\bar u_2,\bar m_2)$ be two solutions of \eqref{eq:ergpb}. Let $\ep>0$, $\xi:\R^d\to \R$ be a smooth, nonnegative, symmetric kernel with a support contained in the unit ball and of integral one. We set $\xi^\ep(x)=\frac{1}{\ep^d} \xi(x/\ep)$ and,  for $i=1,2$,  $m^\ep_i= \xi_\ep\star \bar m_i$ and $\ds V^\ep_i= \frac{\xi^\ep \star (\bar m_iD_pH(\cdot,D\bar u_i))}{m_i^\ep}$. Then 
$\ds -{\rm div} \left( m_i^\ep V^\ep_i\right)= 0$ in $\T^d$. We multiply this equality by $(\bar u_1-\bar u_2)$, integrate by parts, and subtract the resulting formulas to get: 
$$
\int_{\T^d} \lg D(\bar u_1-\bar u_2), m_1^\ep V^\ep_1-m_2^\ep V^\ep_2\rg = 0\;.
$$
Therefore
\be\label{eq:Rep}
\begin{array}{rl}
0\; = & \ds \int_{\T^d} \lg D(\bar u_1-\bar u_2), \xi^\ep\star\left(\bar m_1D_pH(\cdot,D\bar u_1)-\bar m_2D_pH(\cdot,D\bar u_2)\right) \rg \\
= & \ds  \int_{\T^d} \lg D(\bar u_1-\bar u_2), m_1^\ep D_pH(x,D\bar u_1)- m_2^\ep D_pH(x,D\bar u_2) \rg+R_\ep
\end{array}
\ee
where (the double integral being on $\R^d\times \T^d$)
$$
\begin{array}{l}
\ds R_\ep =\\
\ds  \iint \xi^\ep(x-y) \lg D(\bar u_1-\bar u_2)(x)), (D_pH(y,D\bar u_1(y))-D_pH(x,D\bar u_1(x)))\rg \bar m_1(dy) dx \\
\ds -\iint \xi^\ep(x-y) 
 \lg D(\bar u_1-\bar u_2)(x)), (D_pH(y,D\bar u_2(y))-D_pH(x,D\bar u_2(x))\rg \bar m_2(dy)dx\\
=  \ds \iint \xi(z) \lg D(\bar u_1-\bar u_2)(y+\ep z)), (D_pH(y,D\bar u_1(y))-D_pH(y+\ep z,D\bar u_1(y+\ep z)))\rg \bar m_1(dy) dz \\
 \ds -\iint\xi(z) \lg D(\bar u_1-\bar u_2)(y+\ep z)), (D_pH(y,D\bar u_2(y))-D_pH(y+\ep z,D\bar u_2(y+\ep z)))\rg \bar m_2(dy)dz\\
\end{array}
$$
Since $\bar u_1$ and $\bar u_2$ are Lipschitz continuous, we get 
$$
\begin{array}{l}
\ds \left| R_\ep\right| \leq  \\ 
\ds  \ \ds C \iint \xi(z)\left| D_pH(y,D\bar u_1(y))-D_pH(y+\ep z,D\bar u_1(y+\ep z)))\right|  \bar m_1(dy) dz \\
\; \ds + C \iint \xi(z) \left| D_pH(y,D\bar u_2(y))-D_pH(y+\ep z,D\bar u_2(y+\ep z)))\right|  \bar m_2(dy)dz
\end{array}
$$
Following Remark \ref{rem:1}, the maps $D\bar u_i$ are continuous at  $\bar m_i-$a.e.  $y\in \R^d$ for $i=1,2$. Then Lebesgue Theorem implies that $R_\ep\to 0$ as $\ep \to 0$. 

Next we multiply the equality satisfied by the $\bar u_i$ by $(m_1^\ep-m_2^\ep)$, integrate in space (recalling that the $m^\ep_i$ are probability measures) and subtract to get
$$
\int_{\T^d} (m_1^\ep-m_2^\ep)\left[H(x,D\bar u_1)-H(x,D\bar u_2) -F(x,\bar m_1)+F(x,\bar m_2)\right] dx =0\;. 
$$
We combine the above equality with \eqref{eq:Rep} and obtain
$$
\begin{array}{l}
\ds \int_{\T^d} \lg D(\bar u_1-\bar u_2), m_1^\ep D_pH(x,D\bar u_1)- m_2^\ep D_pH(x,D\bar u_2) \rg \\
\qquad \ds - (m_1^\ep-m_2^\ep)\left[H(x,D\bar u_1)-H(x,D\bar u_2) -F(x,\bar m_1)+F(x,\bar m_2)\right] dx = -R_\ep\;. 
\end{array}
$$
Let us set (for $i=1,2$) $H_i= H(x,D\bar u_i)$ and $DH_i= D_pH(x,D\bar u_i)$. Then, following Lasry-Lions classical computation \cite{LL06cr1, LL06cr2,LL07mf} the above formula can be rearranged as 
$$
\begin{array}{r}
\ds \int_{\T^d} m_1^\ep  \left[ H_2-H_1-\lg DH_1, D(\bar u_2-\bar u_1)\rg \right]
+ m_2^\ep  \left[ H_1-H_2-\lg DH_2, D(\bar u_1-\bar u_2)\rg\right]  \\
\ds + \int_{\T^d} (m_1^\ep-m_2^\ep)(F(x,m_1)-F(x,m_2))  = -R_\ep\;. 
\end{array}
$$
Using the convexity of $H$, the terms $H_2-H_1-\lg DH_1, D(\bar u_2-\bar u_1)\rg$ and $H_1-H_2-\lg DH_2, D(\bar u_1-\bar u_2)\rg$ are nonnegative. So
$$
\int_{\T^d} (m_1^\ep-m_2^\ep)(F(x,\bar m_1)-F(x,\bar m_2))  \leq -R_\ep\;.
$$
We let $\ep\to 0$ to get 
$$
\int_{\T^d} (F(x,\bar m_1)-F(x,\bar m_2))d(\bar m_1-\bar m_2)  \leq 0\;,
$$
which in turn implies that
$$
\int_{\T^d} (F(x,\bar m_1)-F(x,\bar m_2))^2dx  \leq 0\;
$$
thanks to assumption \eqref{hyp:Fstrict}. So $F(x,\bar m_1)=F(x,\bar m_2)$. Then $(\bar \lambda_1,\bar u_1)$ and $(\bar \lambda_2, \bar u_2)$  are two solutions of an ergodic problem for the same Hamilton-Jacobi equation. This is known to entail that $\bar \lambda_1=\bar \lambda_2$ (see, for instance, \cite{LPV}). 
\end{proof}

\begin{Example}{\rm Assume for instance that $H$ is quadratic: $$\ds H(x,p)=\frac12 |p|^2-V(x)\ ,$$ where $V:\R^d\to \R$ is a smooth, periodic map. 
Let $(\bar \lambda, \bar u, \bar m)$ be a solution of \eqref{eq:ergpb}. Then,  following \cite{LPV}, 
$$
\bar \lambda = -\min_{x\in \T^d} \left\{V(x)+ F(x, \bar m)\right\}
$$
and $\bar m$ is supported in the set $\ds {\rm argmin}_{x\in \T^d} \left\{V(x)+ F(x, \bar m)\right\}$. Moreover, $D\bar u=0$ in this set. 
}\end{Example}

\section{Convergence}

Let us set $v^T(s,x)=u^T(sT,x)$  and $\nu^T(s,x)= m^T(sT,x)$ for $(s,x)\in [0,1]\times \T^d$. 
Our aim is to prove the uniform convergence of $\ds \frac{v^T(s,\cdot)}{T}$ to $\bar \lambda(1-s)$, where $\bar \lambda$ is the unique ergodic constant associated to the problem \eqref{eq:ergpb}. 

\begin{Theorem}\label{thm:uT/T} Under assumptions \eqref{regucondF}, \eqref{hyp:unifCv} and \eqref{hyp:Fstrict}, the following estimates hold:
$$
\sup_{s\in [0,1]} \left\| \frac{v^T(s,\cdot)}{T}-\bar \lambda (1-s) \right\|_\infty \leq  \frac{C}{T^{\frac12}}
$$
and 
$$
\int_0^1 \|F(\cdot,\nu^T(s))-F(\cdot,\bar m)\|_\infty\ ds \leq \frac{C}{T^{\frac12}}\;.
$$
In particular,  there is a uniform convergence of the map $\ds (s,x)\to \frac{v^T(s,x)}{T}$  to the map $s\to \bar \lambda(1-s)$ as $T\to+\infty$. 
\end{Theorem}

In order to prove Theorem \ref{thm:uT/T}, we need some uniform estimates: 

\begin{Lemma}\label{lem:Esti} The map $u^T$ is uniformly (with respect to $T$)  Lipschitz continuous. 
\end{Lemma}

\begin{proof} Throughout the proof $C$ denotes a constant which varies with the data and may change from line to line. Since $Du^T$ is bounded (with a bound which depends a priori on $T$) and $H$ is of class ${\mathcal C}^2$, the map $t\to m^T(t)$ is Lipschitz continuous for the Monge-Wasserstein distance (see Lemma~\ref{lem:LipschConti} below). We denote by $L_T$ its Lipschitz constant and by $C_0$ be the Lipschitz continuity modulus of $F$ given by assumption~\eqref{regucondF}. We claim that 
\be\label{Lipt}
\left\|\partial_t u\right\|_\infty  \leq \|F\|_\infty+ \|H(\cdot, Du^f)\|_\infty+ C_0L_TT.
\ee
To prove \eqref{Lipt}, we first note that, as the terminal condition $u^f$ being of class ${\mathcal C}^2$, the maps 
$$
(t,x)\to u^f(x)\pm  \left( \|F\|_\infty+ \|H(\cdot, Du^f)\|_\infty\right)(T-t)
$$
are respectively super- (for $+$) and sub- (for $-$) solutions of equation~\eqref{MFG}-(i) with terminal condition $u^f$.  Since $u$ is a solution of this equation, we get by comparison  
\be\label{proofLipt1}
\|u(T-h,\cdot)-u^f\|_\infty \leq \left( \|F\|_\infty+ \|H(\cdot, Du^f)\|_\infty\right) h\qquad \forall h\in (0,T).
\ee
On another hand, since the left-hand side of equation~\eqref{MFG}-(i) is independent of time while the right-hand side is Lipschitz continuous in time, with a Lipschitz constant bounded above by $C_0L_T$,  the maps 
$$(t,x)\to u(t-h,x)\pm \left( \|u(T-h,\cdot)-u^f\|_\infty+ C_0L_Th(T-t)\right)$$ are respectively a super (for $+$) and subsolution (for $-$) of the equation satisfied by $u$. Using again comparison, we have 
$$
|u(t,x)-u(t-h,x)|\leq \|u(T-h,\cdot)-u^f\|_\infty+C_0L_Th(T-t) \qquad \forall (t,x)\in (h, T)\times \T^d. 
$$ 
Plugging \eqref{proofLipt1} into this last inequality then gives \eqref{Lipt}. 

Using the coercivity of $H$ in~\eqref{MFG}-(i), we obtain therefore 
\be\label{Lipx}
\|Du^T\|_\infty^2\leq C(L_T+ 1).
\ee
As $H$ is of class ${\mathcal C}^2$, we have then 
$$
\|D_pH(\cdot, Du^T(\cdot,\cdot))\|_\infty \leq C(L_T+1)^{\frac12}\;.
$$
Then we use Lemma \ref{lem:LipschConti}, which states that $m^T$ is  Lipschitz continuous for the Monge-Wasserstein distance with a Lipschitz constant bounded above by $\|D_pH(\cdot, Du^T(\cdot,\cdot))\|_\infty$. So 
$$
L_T\leq C(L_T+1)^{\frac12}.
$$
This shows that $L_T$ is bounded uniformly in $T$ and, in view  of~\eqref{Lipt} and~\eqref{Lipx} that $u^T$ is Lipschitz continuous with a constant independent of $T$
\end{proof}

\begin{Lemma} [Energy estimate] If $(u^T,m^T)$ is as above and  $(\bar \lambda, \bar u,\bar m)$ is a solution of the ergodic problem \eqref{eq:ergpb}, then
$$
\int_0^T \int_{\T^d} (F(x,m^T(t))-F(x,\bar m))d(m^T(t)-\bar m)dt \leq C
$$
\end{Lemma}

\begin{proof} We use the same kind of argument as for the uniqueness part of Theorem \ref{th:ergopb}. Let   $\ep>0$, $\xi$ and $\xi^\ep(x)=\frac{1}{\ep^d} \xi(x/\ep)$ be as before. We set $m^\ep= \xi_\ep\star m$ and $\ds V^\ep= \frac{\xi^\ep \star (\bar mD_pH(\cdot,D\bar u))}{m^\ep}$, so that $\ds -{\rm div} \left( m^\ep V^\ep\right)= 0$ in $\T^d$. We multiply this equality  by $(u^T(t)-\bar u)$ and integrate on $(0,T)\times \T^d$: 
\be\label{UnifEst1}
\begin{array}{rl}
0\; = & \ds  \int_0^T \int_{\T^d} \lg D(u^T(t)-\bar m), m^\ep V^\ep\rg \\
= & \ds  
\int_0^T \int_{\T^d} \lg D(u^T(t)-\bar m), m^\ep D_pH(x,D\bar u) \rg +R_\ep
\end{array}\ee
where, as in the proof of Theorem \ref{th:ergopb}, $R_\ep\to 0$ as $\ep\to 0$.  Since $m^T(t)$ and $m^\ep$ are probability measures, we also have,  in view of \eqref{eq:ergpb}-(i):
\be\label{UnifEst2}
\int_0^T\int_{\T^d} (m^T(t)-m^\ep) (H(x,D\bar u)-F(x,\bar m)) =0\;.
\ee
From \eqref{MFG}-(ii) we have
$$
\ds \partial_t(m^T-m^\ep)-{\rm div} \left( m^T D_pH(x,Du^T)\right)= 0.
$$
Multiplying this equality by the Lipschitz map $u^T-\bar u$ and integrating in time-space gives, thanks to \eqref{MFG}-(i), 
\be\label{UnifEst3}
\begin{array}{l}
\ds 0= \int_{\T^d}\left[ (u^f-\bar u)(m^T(T)-m^\ep)-(u^T(0)-\bar u)(m_0-m^\ep)\right]\\
\qquad \ds + \int_0^T \int_{\T^d} -(H(x,Du^T(t))-F(x,m^T(t)))(m^T(t)-m^\ep)\\
\qquad \ds + \int_0^T \int_{\T^d} \lg D(u^T(t)-\bar u), m^T(t)D_pH(x,u^T(t))\rg
\end{array}
\ee
Note that the first integral is bounded uniformly with respect to $T$ because 
$$
\left|\int_{\T^d} (u^f-\bar u)(m^T(T)-m^\ep)\right|\leq 2(\|u^f\|_\infty+\|\bar u\|_\infty)
$$
while, since  $m_0$ and $m^\ep$ are probability measures, 
$$
\begin{array}{l}
\ds 
\left|\int_{\T^d} (u^T(0)-\bar u)(m_0-m^\ep)\right|\\
\qquad \leq  \ds  
\left|\int_{\T^d} (u^T(0)-\int_{\T^d}u^T(0))(m_0-m^\ep)\right|+
\left|\int_{\T^d} \bar u(m_0-m^\ep)\right| \\
\qquad \leq  \ds 
2 \|Du^T(0)\|_\infty + 2\|\bar u\|_\infty 
\end{array}
$$
where $Du^T$ is uniformly bounded. Putting together \eqref{UnifEst1}, \eqref{UnifEst2}, \eqref{UnifEst3} and rearranging as in the proof of Theorem \ref{th:ergopb} we obtain
$$
\begin{array}{l}
\ds \int_0^T\int_{\T^d} m^T(t)  \left[ \bar H-H(t)-\lg DH(t), D(\bar u- u^T)\rg \right] \\
\qquad \ds +  \int_0^T\int_{\T^d}  m^\ep  \left[ H(t)-\bar H-\lg D\bar H, D(u^T(t)-\bar u)\rg\right]\qquad  \\
\qquad \ds +  \int_0^T\int_{\T^d} (m^T(t)-m^\ep)(F(x,m^T(t))-F(x,\bar m))  \leq C +R_\ep
\end{array}
$$
where we have set $H(t)= H(x,Du^T(t,x))$, $DH(t)= D_pH(x, Du^T(t,x))$, $\bar H= H(x,D\bar u(x))$ and $D\bar H= D_pH(x,D\bar u(x))$. The first two terms being nonnegative, we get the desired result by letting $\ep\to 0$. 
\end{proof}

\begin{proof}[Proof of Theorem \ref{thm:uT/T}] Recall the notations $v^T(s,x)= u^T(sT,x)$  and $\nu^T(s,x)= m^T(sT,x)$ for $(s,x)\in [0,1]\times \T^d$. 
According to Lemma \ref{lem:Esti}, we have 
$$
\int_0^1 \int_{\T^d} (F(x,\nu^T(s))-F(x,\bar m))d(\nu^T(s)-\bar m) \leq \frac{C}{T}\;.
$$
From assumption \eqref{hyp:Fstrict} this implies that 
$$
\int_0^1 \int_{\T^d} (F(x,\nu^T(s))-F(x,\bar m))^2\ dxdt \leq \frac{C}{T}\;,
$$
and, using the uniform regularity of the map $F$ and Hölder inequality, 
\be\label{conv:F}
\begin{array}{rl}
\ds \int_0^1 \|F(\cdot,\nu^T(s))-F(\cdot,\bar m)\|_\infty\ ds \; \leq & \ds C \int_0^1 \|F(\cdot,\nu^T(s))-F(\cdot,\bar m)\|_2\ ds\\
 \leq & \ds \frac{C}{T^{\frac12}}\;.
\end{array}
\ee
Note that the map $v^T$ solves 
\be\label{eq:vT}
-\frac{\partial_s v^T}{T} +H(x,Dv^T) =F(x,\nu^T), \qquad v^T(1,x)=u^f(x)\;, 
\ee
while the map $w^T(s,x)= \bar u(x)+T\bar \lambda(1-s)$ solves 
$$
-\frac{\partial_s w^T}{T} +H(x,Dw^T) =F(x,\bar m), \qquad v^T(1,x)=\bar u(x)\;. 
$$
From standard estimates in viscosity solutions we deduce that, for any $t\in [0,1]$, 
$$
\begin{array}{rl}
\ds \left\| v^T(t,\cdot)-w^T(t,\cdot)\right\|_\infty \; \leq & \ds \left\| v^T(1,\cdot)-w^T(1,\cdot)\right\|_\infty \\
& \qquad \ds +T \int_t^1 \|F(\cdot,\nu^T(s))-F(\cdot,\bar m)\|_\infty\ ds\;.
\end{array}
$$
So
$$
\begin{array}{rl}
\ds \left\| \frac{v^T(t,\cdot)}{T}-\bar \lambda (1-t) \right\|_\infty \; \leq & \ds \frac{\|u^f\|_\infty+2\|\bar u\|_\infty}{T} 
+   \int_t^T \|F(\cdot,\nu^T(s))-F(\cdot,\bar m)\|_\infty ds\\
\leq & \ds  \frac{C}{T^{\frac12}}\;.
\end{array}
$$
\end{proof}

%
%
%
%
%

\section{Appendix: proof of the existence and uniqueness result}

The following result is stated in \cite{LL07mf}. For convenience of the reader we recall the main ingredients of proof. 

\begin{Theorem}[\cite{LL07mf}]\label{thm:MFG} Let $H$ and $F$ satisfy conditions  \eqref{regucondF} and \eqref{hyp:unifCv}. Then equation \eqref{MFG} has a solution. If moreover the following inequality holds:
$$
\int_{\T^d} (F(x,m_1)-F(x,m_2))d(m_1-m_2) \geq 0\qquad \forall m_1,m_2\in P(\T^d),
$$
then the solution of \eqref{MFG} is unique. 
\end{Theorem}
 
\begin{proof}
The proof is based on a vanishing viscosity argument. Let $\ep>0$ and $(u^\ep,m^\ep)$ be the solution to 
\be\label{MFGsto}
\left\{\begin{array}{cl}
(i)& -\partial_t u^\ep -\ep\Delta u^\ep +H(x,Du^\ep) =F(x,m^\ep(t))\; {\rm in }\; (0,T)\times \R^d\\
\m
(ii) & \partial_t m^\ep-\ep\Delta m^\ep -{\rm div} (m^\ep D_pH(x,Du^\ep))=0\; {\rm in }\; (0,T)\times \R^d\\
\m
(iii)& m^T(0)=m_0, \; u^\ep(x,T)=u^f(x)\; {\rm in}\; \R^d
\end{array}\right.
\ee
By standard regularity results for parabolic equations and fixed point arguments, it is not difficult to check that system \eqref{MFGsto} has at least one classical solution. Moreover, as $u^f$ is of class ${\mathcal C}^2$ and the Hamiltonian satisfies~\eqref{hyp:unifCv}, the map $u^\ep$ is semiconcave with a semiconcavity argument independent of $\ep$. In particular $u^\ep$ is uniformly Lipschitz continuous. Let us now show that $m^\ep$ is uniformly bounded. For this we note that 
$$
\begin{array}{l}
\ds {\rm div} (m^\ep D_pH(x,Du^\ep))\\
\ds \qquad = \lg Dm^\ep, D_pH(x,Du^\ep)\rg + m^\ep {\rm Tr}\left(D^2_{xp}H(x,Du^\ep)+ D^2_{pp}H(x,Du^\ep)D^2u^\ep\right)\\
\ds \qquad \leq \lg Dm^\ep, D_pH(x,Du^\ep)\rg + Cm^\ep
\end{array}
$$
because $D^2_{xp}H(x,Du^\ep)$ is bounded thanks to the regularity of $H$ and the uniform Lipschitz continuity of $u^\ep$, and 
${\rm Tr}\left(D^2_{pp}H(x,Du^\ep)D^2u^\ep\right)$ is bounded above because $D^2_{pp}H$ is positive and $u^\ep$ is uniformly semiconcave. 
So $m^\ep$ is a subsolution of  the transport equation
$$
\partial_t m^\ep-\ep\Delta m^\ep -\lg Dm^\ep, D_pH(x,Du^\ep)\rg -Cm^\ep=0\; {\rm in }\; (0,T)\times \R^d.
 $$
By maximum principle we get $\|m^\ep\|_\infty \leq \|m_0\|_\infty e^{CT}.$ Next we claim that the map $t\to m^\ep(t)$ is uniformly Hölder continuous: indeed, if we multiply \eqref{MFGsto}-(ii) by $m^\ep$ and integrate, we get: 
 $$
 \ep \int_0^T\int_{\T^d} |Dm^\ep|^2 \leq \int_{\T^d} \left|m^2(T)-m^2(0)\right| + \int_0^T\int_{\T^d} m^\ep|Dm^\ep||D_pH(x,Du^\ep)|. 
 $$ 
As $m^\ep$ and $D_pH(x,m^\ep)$ are uniformly bounded, this implies that:
$\ds  \ep \left(\int_0^T\int_{\T^d} |Dm^\ep|^2\right)^{\frac12}\leq C.$
 Then,  for any smooth test function $\varphi$ and for any $0\leq t_1\leq t_2\leq T$, we have by~\eqref{MFGsto}-(ii) that 
$$
\int_{\T^d} \varphi(m^\ep(t_2)-m^\ep(t_1)) = - \int_{t_1}^{t_2}\int_{\T^d} \ep\lg Dm^\ep, D\phi\rg + m^\ep\lg D_pH(x,Du)), D\phi\rg \leq
C(t_2-t_1)^{\frac12}\|D\varphi\|_\infty
$$
because $Du^\ep$ and $m^\ep$ are uniformly bounded. Taking the supremum over all  $1-$Lipschitz map  $\varphi$ gives ${\bf d}_1(m(t_1),m(t_2))\leq C(t_2-t_1)$. 

 Because of the bounds on $(u^\ep,m^\ep)$, we can assume that (up subsequences) $u^\ep$ converges uniformly to some Lipschitz continuous map $u$. On another hand $m^\ep$ converges  in $L^\infty-$weak*  and in ${\mathcal C}^0([0,T],P(\T^d))$ to some $m\in L^\infty\cap {\mathcal C}^0([0,T],P(\T^d))$. In particular, $m(0)=m_0$. Using the continuity assumption~\eqref{regucondF}, we also have that $F(\cdot, m^\ep(\cdot))$ converges uniformly to $F(\cdot, m(\cdot))$. By standard viscosity solutions arguments, we can conclude to the convergence of $u^\ep$ to the unique solution of 
 $$
 \left\{\begin{array}{l}
-\partial_t u+H(x,Du) =F(x,m(t))\; {\rm in }\; (0,T)\times \R^d\\
u(x,T)=u^f(x)\; {\rm in}\; \R^d
\end{array}\right.
$$
Next we turn to the limit of $m^\ep$:  for a fixed test function $\varphi\in{\mathcal C}^\infty_c((0,T)\times \T^d)$, we have 
$$
\int_0^T\int_{\T^d}\left(-\partial_t \varphi-\ep \Delta \varphi +\lg D_pH(x,Du_n),D\varphi(t,x)\rg \right) m_n(t,x)=0
$$
where $D_pH(x,Du_n)$ is bounded and converges a.e. to $D_pH(x,Du)$ while $m^\ep$ converges weakly* to $m$. So we get as $\ep\to+\infty$, 
$$
\int_0^T\int_{\T^d}\left(-\partial_t \varphi(t,x)+\lg D_pH(x,Du),D\varphi(t,x)\rg \right) m(t,x)=0,
$$
which shows that $m$ is a solution of the continuity equation \eqref{MFG}-(ii). In conclusion, the pair $(u,m)$ solves \eqref{MFG}. The uniqueness for this system is established in full details in \cite{LL07mf}, so we omit the proof. 
\end{proof}

We complete the paper by a standard estimate on the continuity equation: 
\be\label{ContEq}
\partial_t m +{\rm div}(m b)=0 \; {\rm in }\; (0,T)\times \T^d
\ee

\begin{Lemma}\label{lem:LipschConti} Assume that $b:(0,T)\times \T^d\to\R^d$ is a Borel vector field with $\|b\|_\infty<+\infty$. If $m$ satisfies~\eqref{ContEq}, then $m$ is  Lipschitz continuous as a map from  $[0,T]$ to $P(\T^d)$, with a Lipschitz constant bounded above by $\|b\|_\infty$.
\end{Lemma}

\begin{proof}  Fix $0<  t_1< t_2< T$ and let $h\in {\mathcal C}^\infty_c(\T^d)$ be $1-$Lipschitz continuous. Let  $\ep>0$ small and
$$
\varphi_\ep(t,x)=\left\{ \begin{array}{ll}
 (t-t_1)h(x)/\ep & {\rm if }\;  t\in [t_1,t_1+\ep]\\
 h(x) & {\rm if } \;  t\in [t_1+\ep,t_2-\ep]\\
(t_2-t)h(x)/\ep & {\rm if } \;  t\in [t_2-\ep,t_2]\\
0 & {\rm otherwise}
\end{array}\right.
$$
As 
$$
 \int_0^T \int_{\R^d} \left(-\partial_t\varphi_\ep +\lg m, D\varphi_\ep\rg\right)m  = 0,
 $$
we have 
$$
  \int_{t_1}^{t_1+\ep}  \int_{\R^d} \frac{hm }{\ep}  +\int_{t_1+\ep}^{t_2-\ep} \int_{\R^d} \lg b, Dh\rg m  +   \int_{t_2-\ep}^{t_2}  \int_{\R^d} -\frac{hm}{\ep}  = o(1).
$$
Letting $\ep\to 0$ gives for a.e. $0<t_1<t_2<T$: 
$$
   \int_{\R^d} h\ d(m(t_1)-m(t_2))   +\int_{t_1}^{t_2} \int_{\R^d} \lg b, Dh\rg m=0\;.
$$
So
$$
   \int_{\R^d} h\ d(m(t_1)-m(t_2))  \leq \|b\|_\infty\| Dh\|_\infty |t_2-t_1|\;.
$$
 Taking the sup over $h$ gives then:
 $$
 {\bf d}_1(m(t_1),m(t_2)) \leq   \|b\|_\infty |t_2-t_1|\;.
 $$
\end{proof}

\end{document}